\documentclass[a4paper,11pt]{article}

\title{On some planar Baumslag-Solitar actions}
\author{Juan Alonso, Nancy Guelman, Crist\'obal Rivas and Juliana Xavier}
\date{}

\usepackage{diagbox}
\usepackage{graphicx,psfrag,subfigure}
\usepackage{fancyhdr}
\usepackage{mathrsfs}
\usepackage{verbatim}
\usepackage[ansinew]{inputenc}
\usepackage{amsmath,amsthm,amssymb,amscd}

\newcommand{\diff}{\mbox{Diff}}
\newcommand{\N}{\mathbb{N}}
\newcommand{\Z}{\mathbb{Z}}

\newcommand{\R}{\mathbb{R}}

\newcommand{\T}{\mathbb{T}}

\DeclareMathOperator{\fix}{Fix}

\DeclareMathOperator{\id}{Id}

\newtheorem{rem}{Remark}

\newtheorem{teo}{Theorem}[section]
\newtheorem{cor}[teo]{Corollary}
\newtheorem{lema}[teo]{Lemma}

\theoremstyle{definition}

\newtheorem{obs}[teo]{Remark}

\theoremstyle{remark}
\makeindex

\begin{document}

\maketitle

\begin{abstract} Let $BS(1,n)= \langle a,b : a b a ^{-1} = b ^n\rangle$ be the solvable Baumslag-Solitar group for $n \geq  2$. We study representations of $BS(1, n)$ on the plane by
orientation preserving homeomorphisms, assuming that $a$ acts as a linear map and $b$ as a map with bounded displacement. We find that the possibilities for a faithful action depend greatly on the Jordan canonical form of the map $h$ defined by the action of $a$. In case $h$ is diagonalizable over $\R$, we shall give examples or prove rigidity theorems depending on the eigenvalues. We also show some rigidity in the cases where $h$ is elliptic or parabolic. Then we give applications to the actions of $BS(1, n)$ by homeomorphisms of the torus.


\end{abstract}

\begin{section}{Introduction}

For $n, m\in \Z \backslash \{0\}$, the Baumslag-Solitar group $BS(m,n)$ is defined by the presentation $BS(m,n)= \langle a,b : a b^m a ^{-1} = b ^n\rangle$. These groups were introduced by Baumslag and Solitar in \cite{bsp} to provide the first examples of two generator non-Hopfian groups with a single defining relation. The groups $BS(1,n)$ for $n \geq  2$ are the simplest examples of infinite non-abelian solvable groups, and also provide examples of distorted elements, which are related to Zimmer's conjecture \cite{zimmer} and play an important role in surface dynamics \cite{fh}.

Recently there has been an interest on understanding the possible dynamics of Baumslag-Solitar group actions on surfaces, the results usually showing restrictions to these types of actions. J. Franks and M. Handel \cite{fh} proved that on a surface $S$ of genus greater than one, any distortion element in the group $\diff^1_0(S,area)$  is a torsion element, 
and therefore there are no faithful representations of $BS(1,n)$ in $\diff^1_0(S,area)$.  N. Guelman and I. Liousse \cite{glp} constructed a smooth $BS(1, n)$ action without finite orbits on $\T^2$ that is not locally rigid, but on the other hand, they proved that there are no minimal faithful actions of $BS(1, n)$ on $\T^2$. J. Alonso, N. Guelman
and J. Xavier \cite{agx} proved that there are no faithful representations of $BS(1, n)$ on surfaces where $a$ acts by a (pseudo-)Anosov homeomorphism with stretch factor $\lambda > n$. In the case of the torus, they showed that there are no faithful actions where $a$ acts by an Anosov map and $b$ by an area-preserving map. This last result also holds without the area preserving hypothesis if the action is $C^1$, as shown by Guelman and Liousse in \cite{gl2}. They also show that any faithful action of $BS(1, n)$ on a closed surface where the action of $b$ is $C^1$ and $a$ acts by a (pseudo-)Anosov map $h$ has a finite orbit contained in the set of singularities of $h$. In the general scope, these works lie in the context of studying the obstructions to the existence of faithful group actions on certain spaces (a survey of these ideas can be found in  \cite{fisher}). Another result in this direction is that of S. Hurtado and J. Xue in \cite{hurtado}, concerning Abelian-by-cyclic groups, which generalize the groups $BS(1,n)$. They found that if such a group acts on the torus $\T^2$ by $C^r$-diffeomorphisms, with $r>2$, non-finite image and containing an Anosov map, then the action is topologically conjugated to a linear action.

In this paper we mainly consider actions of $BS(1,n)$, for any $n\geq 2$,  on the plane by orientation preserving homeomorphisms. To give such an action is the same as giving $f,h:\R^ 2 \to \R^2$ orientation preserving homeomorphisms satisfying 
\begin{equation}\label{BS eq} hfh^{-1} = f ^n 
\end{equation} namely, $f$ and $h$ are the respective images of $b$ and $a$ under the representation. We compensate for the lack of compactness by supposing that the distance from $f$ to the identity map is uniformly bounded, that is, 
that there exists $K>0$ such that $$|f(x)-x|\leq K 	\text{ for all $x\in \R^ 2$}$$ We say that $f$ has {\em bounded displacement} if this condition holds. Note that this is the case when $f$ is the lift of a map of the torus that is isotopic to the identity. The main motivation for this hypothesis is to consider actions on the plane that are lifted from certain actions on the torus, as we shall see in Section \ref{s:toro}. 

In this work we will be interested in the case when $h$ is a linear map. Observe that if we conjugate the action by $A\in GL(2,\R)$ we get another action in our hypotheses, since $AfA^{-1}$ has bounded displacement with constant $||A||K$. Thus we may assume that the associated matrix of $h$ is a canonical form over $\R$. First we focus on the case when $h$ is diagonalizable over $\R$, thus assuming that its associated matrix is 
$$D=\left(\begin{array}{cc} \lambda & 0 \\ 0 & \mu\end{array}\right) \quad \mbox{for }\lambda,\mu\in \R$$ where $\lambda\mu>0$ since $h$ is an orientation preserving homeomorphism. Excluding the cases when either $|\lambda|$ or $|\mu|$ equals $1$, our results can be summarized as follows, where actions are assumed to be faithful:

\begin{center}
    \begin{tabular}{ | c | p{3.2cm} | p{3.2cm} | p{3.2cm} |}
    \hline
    \diagbox{$\lambda$}{$\mu$} & $ 0<|\mu|<1$ & $1<|\mu|\leq n$ & $n<|\mu|$ \\ \hline
    $0<|\lambda|<1$ & No action: \newline Theorem \ref{teo contractivo} \vspace{.5cm}& Product action  & No action: \newline Lemma \ref{lem 3,1}\\ \hline
    
    $1<|\lambda| \leq n$ & Product action \vspace{1cm} & Strange Rotation action and Product action & Preserves horizontal foliation: \newline Corollary \ref{nucero-b}. \newline Product action. \vspace{.5cm} \\ \hline
    
    $n<|\lambda| $ & No action: \newline Lemma \ref{lem 3,1} & Preserves vertical foliation: \newline Corollary \ref{nucero-b}. \newline Product action.\vspace{.5cm} &No action: \newline Lemma \ref{lem 3,3} \\
    \hline
    \end{tabular}
\end{center}

That is, we give examples or prove rigidity theorems depending on the eigenvalues of $h$. The rigidity theorems are either non-existence of faithful actions, or preservation of a foliation of lines, as indicated in the table. The examples show when such rigidity statements do not hold, namely, the product action (Section \ref{sec product action}) is a faithful action that preserves a foliation by straight lines, while the strange rotation action (Section \ref{sec strange rot}) is a faithful action that does not preserve such a foliation. If one of the eigenvalues is $\pm 1$, say $|\lambda|=1$, then  $|\mu|\neq 1$ for $h$ to have infinite order. In this case, the construction of the product action for $1<|\mu|\leq n$ still holds, and we give an example of a faithful action for $|\lambda|=1$ and $|\mu|<1$ in Section \ref{sec ejemplo con vp 1}.  For $|\lambda|=1$ and $|\mu|>n$ there are no faithful actions, by Lemma \ref{lem lambda=1}.

For $h\in SL(2,\R)$ we study the other possible canonical forms. In the elliptic case we show non-existence of faithful actions, in Corollary \ref{noac-elip}. For the parabolic case we show that $f$ fixes every point in the eigenspace of $h$ (Corollary \ref{cor-par}). In particular, if the eigenvalue of $h$ is $1$, then its eigenspace is $\fix(h)$, and hence it is the set of global fixed points of the action. This will allow us to prove that there are no faithful actions of $BS(1,n)$ on the torus where $a$ acts as a Dehn twist.

\end{section}

\begin{section}{Preliminaries} \label{s:prelim}

We shall use some results from surface homeomorphism theory that we state in this section. For an homeomorphism $f$ we denote by $\fix(f)$ its set of fixed points, and by $\Omega(f)$ its non-wandering set, which, we recall, contains the recurrent points of $f$ and in particular its periodic points. The following is the classic Brouwer's Theorem on planar homeomorphisms:

\begin{teo}\label{brou} Let $f:\R^2\to \R^2$ be  an orientation preserving homeomorphism with $\Omega(f)\neq\emptyset$. Then $f$ has a fixed point.
 
\end{teo}

Let $f$ be a surface homeomorphism. A {\em periodic disk chain} for $f$  is a finite set $U_0, U_1, \ldots U_{n-1}$ of pairwise disjoint topological disks such that for all $i=0, \ldots, n-1$ we have $f(U_i)\cap U_i = \emptyset$, and there exists $m_i$ such that $f^{m_i} (U_i)\cap U_{i+1} \neq \emptyset$, where indexes are taken modulo $n$. Below is Franks's \cite{franks} adaptation of Brouwer's theorem to periodic disk chains: 

\begin{teo}\label{dc} Let $f:\R^2\to \R^2$ be  an orientation preserving homeomorphism which possesses a periodic disk chain.  Then $f$ has a fixed point.
 
\end{teo}

Let $f: A \to A$ be an orientation preserving homeomorphism of the open annulus  $A= S^1\times (0,1)$, and $\tilde f$ a lift of $f$ to the covering space $\tilde A = \R \times (0,1)$. We will say that there is a {\em positively returning disk} for $\tilde f$ if there is an open topological disk $U\subset \tilde A$  such that $\tilde f (U)\cap U= \emptyset$ and $\tilde f ^n (U) \cap (U+k)\neq \emptyset$ for some $n, k > 0$, where $U + k$
denotes the set $\{(x+k, t ): (x, t )\in U\}$.  A {\em negatively
returning disk} is defined similarly but with $k < 0$.
The following result is Theorem (2.1) in \cite{franks}:

\begin{teo}\label{fr} Suppose $f: A\to A$ 
is an orientation preserving homeomor-
phism of the open annulus which is homotopic to the identity,
and satisfies the
following conditions:
\begin{enumerate}
 \item  every point of $A$ is non-wandering,

\item $f$ has at most finitely many fixed points, and

\item there is a lift of $f$ to its universal covering space $\tilde f : \tilde A\to \tilde A$ which
possesses both a positively and a negative returning disk, both of which are lifts of some disks in $A$.
\end{enumerate}
Then $f$ has a fixed point.
 
\end{teo}

We say that a homeomorphism is {\it recurrent} if there exists a sequence $n_k$, $k\to \infty$, such that $d(f^{n_k}, \id)\to 0$ uniformly. Next we state a result on sphere homeomorphisms due to Kolev and P\'erou\`eme \cite{kp}: 

\begin{teo}\label{recu}  Let $f:S^2\to S^2$ be a non trivial recurrent orientation-preserving homemorphism.  Then $f$ has exactly two fixed points.
\end{teo}



\end{section}

\begin{section}{Basic computations} \label{s:basic}

Throughout most of the paper we consider a planar action of $BS(1,n)$  for some $n\geq 2$, given by  orientation preserving homeomorphisms $f,h:\R^2\to\R^2$ satisfying the group relation $hfh^{-1}=f^n$. The following is an easy consequence of the group relation.

\begin{lema}\label{bsp} Let $f$ and $h$ be bijections such that $hfh ^{-1} = f ^n$.  Then:
\begin{enumerate}
 \item $h^m f h ^{-m} = f ^{n^m} $ for all $m\geq 1$;
 \item $h^{-1} (\fix (f))\subset \fix (f)$.
\end{enumerate}
\end{lema}

Fix an isotopy $(f_t)_{t\in [0,1]}$ such that $f_0 = \id$ and $f_1 = f$.  For all $x\in \R^ 2$ we denote by $\gamma _x$ the arc $t\mapsto f_t (x), t\in [0,1]$.  More generally, for
$x\in \R^ 2$ and $m\geq 1$ define $\gamma ^m_x = \prod _{i= 0} ^{m-1}
\gamma _{f^i (x)}$, where the product is concatenation of arcs.\\

The next lemmas relate to transverse measured foliations that have some compatibility with the action. Let $\nu$ be a transverse (signed) measure to an oriented foliation of the plane $\R^ 2$, which can be extended to measure any curve with the following properties:  
\begin{itemize}
\item $\nu (\beta  \gamma) = \nu (\beta) + \nu (\gamma)$, and
\item $\nu(\beta) = \nu (\gamma)$ if $\beta$ and $\gamma$ have the same endpoints.
\end{itemize}
We assume $\nu$ satisfies the following further conditions, that relate to the action:
\begin{itemize}
 \item $\nu (h(\gamma))= a \nu (\gamma)$, for some $a\in\R$, $a\neq 0$,
 
 \item the map $x \to \nu (\gamma _x)$ is bounded on $\R ^ 2$.
\end{itemize}

Note that the first item implies that $h$ preserves the foliation, i.e. takes leaves to leaves. There are general consequences in the cases when $|a|>n$ and when $|a|<1$, which shall lead to our rigidity results. 

\begin{lema}  If $|a|>n$, then $\nu (\gamma _x )\equiv 0$.
 
\end{lema}

\begin{proof}  Note that the arcs $h^ m \gamma _{h^ {-m}(x)}$ and $\gamma ^ {n^m}_x$ have the same endpoints by item 1 in Lemma \ref{bsp}.  Therefore, $$|a|^ m |\nu (\gamma _{h^ {-m}(x)})| = |\nu (h^ m \gamma _{h^ {-m}(x)})| = |\nu (\gamma ^ {n^m}_x)| = \left|\sum _{i=0}^{n^m-1}\nu (\gamma _{f^i(x)})\right|\leq Cn^ m ,$$
where $C$ is the bound for the map $x\to \nu (\gamma_x)$. So,
$$|\nu (\gamma _{h^ {-m}(x)})|\leq \left(\frac{n}{|a|}\right)^ m C,$$ for all $m\geq 1$ and all $x\in \R^ 2$. Since the right term of the last equation is independent of $x$, if $|a|>n$ we have that $\nu (\gamma _x )\equiv 0$.
 
\end{proof}

\begin{cor}\label{nucero} If $\cal F$ is the foliation of $\R^2$ that has $\nu$ as a transverse measure, and $|a|>n$, then each leaf of ${\cal F}$ is $f$-invariant.

 
\end{cor}

\begin{lema}\label{rec}  For all $m\geq 1$, we have $|\nu (\gamma ^ {n^m}_x)|  \leq C |a|^ m$.  In particular, if $|a|<1$, then $\lim_{m\to \infty} \nu (\gamma_x^ {n^m}) = 0$ uniformly on $x$.  
 
\end{lema}

\begin{proof} As before, $$|\nu (\gamma ^ {n^m}_x)| = |a|^ m |\nu (\gamma _{h^ {-m}(x)})| \leq C |a|^ m.$$
 
\end{proof}

\end{section}


\begin{section}{Rigidity in the diagonal case} \label{s:diag rigidity}

Let $f,h:\R^2\to\R^2$ be orientation preserving homeomorphisms with $hfh^{-1}=f^n$, and assume that $f$ has bounded displacement and $h$ is the linear map whose associated matrix is
$$D=\left(\begin{array}{cc} \lambda & 0 \\ 0 & \mu\end{array}\right)$$ i.e. $h(x,y)=(\lambda x,\mu y)$. We devote this section to proving the rigidity results announced in the table shown in the Introduction. In the next section we construct examples showing that the hypotheses of these results cannot be relaxed. 
                                                                  
Observe that $dx$ is a transverse measure to the foliation of $\R^2$ by vertical lines, and that it satisfies the conditions given in Section \ref{s:basic} with $a=\lambda$. The same is true for $dy$ and the horizontal foliation, with $a=\mu$. Then Corollary \ref{nucero} gives us that

\begin{cor} \label{nucero-b} $ $
\begin{itemize}
\item If $|\lambda|>n$, the vertical foliation is preserved by $\langle f,h\rangle$.
\item If $|\mu|>n$, the horizontal foliation is preserved by $\langle f,h\rangle$.

\end{itemize}
\end{cor}

Next we turn to non-existence of faithful actions under certain conditions on the eigenvalues, where we actually show something stronger, namely that $f$ must be the identity map.

\begin{lema}  If  $|\lambda|>n$ and $|\mu| > n$, then $f=\id$.
 \label{lem 3,3}
\end{lema}

\begin{proof}   Using Corollary \ref{nucero} with $a =\lambda$ and $\nu=dx$, one obtains that each vertical line  is preserved by $f$. So
each $f$-orbit is contained in a vertical line, and using Corollary \ref{nucero} with $a =\mu$ and $\nu=dy$, one obtains that each horizontal line is also preserved by $f$. The result follows.

\end{proof} 

\begin{lema}  If  $|\lambda|>n$ and $|\mu| <1$, then $f=\id$.
 \label{lem 3,1}
\end{lema}

\begin{proof}  As in the previous lemma, each $f$-orbit is contained in a vertical line. Thus for any $x\in\R^2$, applying Lemma \ref{rec} with $\nu=dy$ shows that $x$ is a recurrent point for the restriction of $f$ to the vertical line containing $x$. So $x$ must be fixed by $f$, since line homeomorphisms do not have recurrent points that are not fixed.
 
\end{proof}

\begin{lema}  If  $|\lambda|=1$ and $|\mu| >n$, then $f=\id$.
 \label{lem lambda=1}
\end{lema}

\begin{proof}
First we assume that $\lambda=1$, $\mu>n$. Using Corollary \ref{nucero} with $a =\mu$ and $\nu=dy$, we get that $f$ preserves each horizontal line. Moreover, since it has bounded displacement, $f$ preserves the orientation of each horizontal line. Let $L=\R\times\{0\}$ be the horizontal axis, and note that $h|_L=\id_L$ since $\lambda=1$. Then we deduce that $f|_L^{n-1}=\id_L$, and since $f|_L$ is an orientation preserving line homeomorphism, we must have $f|_L=\id_L$.

Suppose $f\neq\id$, and let $p=(x_0,y_0)$ be a point with $f(p)\neq p$. We write the restriction of $f$ to the horizontal by $p$ as $$f(x,y_0)=(\phi(x),y_0) $$ where $\phi:\R\to\R$ is an orientation preserving homeomorphism. Using that $f=h^{-k}f^{n^k}h^k$ for $k\geq 0$ we see that $$f(x,\mu^{-k}y_0)=(\phi^{n^k}(x),\mu^{-k}y_0) $$
and since $(x_0,\mu^{-k}y_0)\to_k(x_0,0)$ and $f$ fixes $(x_0,0)$,    we get by continuity that $\phi^{n^k}(x_0)\to_k x_0$. This is absurd, as an orientation preserving line homeomorphism cannot have a recurrent point that is not fixed.

For the case $\lambda=-1$, $\mu<-n$, we consider the action $\langle f,h^2 \rangle$ of $BS(1,n^2)$. Observing that $h^2$ has eigenvalues $1$ and $\mu^2>n^2$, we get that $f=\id$ from the previous case.
\end{proof}

\begin{teo}\label{teo contractivo} If $|\lambda|<1$ and $|\mu| <1$, then $f= \id$.
 
\end{teo}

We shall prove this theorem through a series of lemmas. We provide two proofs, one reliant on the work of Oversteegen and Tymchatyn \cite{ot} classifying recurrent homeomorphisms of the plane, and the other on the theorems of Franks that are stated in Section \ref{s:prelim}.


\begin{lema}\label{esrec} $f$ is recurrent. Moreover, we have $d(f^ {n^m}, \id)\to_m 0$ uniformly.
 
\end{lema}

\begin{proof} By Lemma \ref{rec},  $$dx (\gamma ^ {n^m}_x) \to _m 0,$$ and $$dy (\gamma ^ {n^m}_x) \to _m 0$$ uniformly.  The result follows.
 
\end{proof}

\begin{obs} It follows from \cite{ot} that $f$ is periodic. This, combined with the fact that nontrivial periodic homeomorphisms have unbounded displacement gives a proof of Theorem \ref{teo contractivo}.  
\end{obs}

We include an alternative proof below, which is a nice
application of some theorems of Franks relying on Brouwer theory.

\begin{lema} \label{fix cero} If $f\neq\id$ then $\fix (f) = \fix (h) = \{(0,0)\}$.
 
\end{lema}

\begin{proof}   By Theorem \ref{recu}, if $f\neq \id$, then $f$ has exactly one fixed point $x$ (we are extending $f$ to $S^2$ fixing $\infty$).  As the set $\fix(f)$ is $h^{-1}$- invariant (by Lemma 
\ref{bsp} item 2.), $x$ must be fixed also by $h$, which implies that $x=(0,0)$.  

\end{proof}

By Lemma \ref{fix cero} both $h$ and $f$ map the annulus $A= \R^2\setminus \{(0,0)\}$ onto itself, thus defining an action of $BS(1,n)$ on $A$. Moreover, the restriction $f|_A:A\to A$ is isotopic to the identity, since it preserves orientation and both ends of the annulus. So we can assume that for $x\in A$ the curves $\gamma _x$ are contained in $A$. We will need the following general statement for annulus maps, which is a consequence of Theorem \ref{fr}:

\begin{lema} \label{franks adapt} Let $g$ be a homeomorphism of the open annulus $S^ 1\times (0,1)$ that is isotopic to the identity and such that $\fix (g) = \emptyset$. Suppose that $x\in S^ 1\times (0,1)$ is a recurrent point for $g$, and let $G$ be a lift of $g$ to the universal cover  $\R\times (0,1)$.  If $(n_m)_{m\in \N}$ is any sequence such that $g^ {n_m}(x)\to x$, then for any lift $\tilde x$ of $x$ we have that $$ \mbox{either }\, 
(G^{n_m}(\tilde x))_1\to + \infty \,\,\mbox{ or }\,\, (G^{n_m}(\tilde x))_1\to - \infty$$ where the subindex $1$ stands for the projection onto the first coordinate of $\R\times(0,1)$.
 
\end{lema}

\begin{proof}
Let $p:\R\times (0,1)\to S^1\times (0,1)$ be the covering projection. Consider $U\subset S^1\times (0,1)$ a neighbourhood of $x$ such that $p^ {-1}(U)= \cup_{\alpha} U_{\alpha}$, each $U_{\alpha}$ projecting homeomorphically onto $U$. Taking a smaller neighbourhood if necessary, we may assume that $g(U)\cap U =\emptyset$, since $g$ has no fixed points. Let $\tilde U$ be the connected component of $ p^{-1}(U)$ that contains $\tilde x$.  Since $g^ {n_m}(x)\to x$, for any sufficiently large $m$ there exists $k_m \in \Z$ such that $G^ {n_m} (\tilde x)\in \tilde U + k_m$, in the notation of Section \ref{s:prelim}. Note that it suffices to prove that either $k_m\to_m \infty$ or $k_m\to_m -\infty$.

Since $g$ has no fixed points, the same is true for $G$, and by Theorem \ref{dc} we see that $G$ has no periodic disk chains, which implies that $k_m\neq 0$ for all $m$.  Moreover, the sign of $k_m$ is constant for all $m$, for otherwise $g$ would have a fixed point by Theorem \ref{fr}. 
Applying Theorem \ref{fr} again, this time for the returning disks $\tilde U+k_m$, we see that either $k_{m+1}-k_m > 1$ for all $m$, or $k_{m+1}-k_m < -1$ for all $m$. The result follows.
\end{proof}

Next we show that every $f$-orbit must turn infinitely many times around this annulus. To give a precise statement, we consider $d\theta$ where $\theta$ is the angular coordinate on $A= \R^2\setminus \{(0,0)\}$, which is well defined on $A$, and is a transverse measure to the foliation of $A$ by radial lines. Then we show the following.

\begin{lema}\label{br} Assume that $f\neq\id$. Then for all $x\in A$, we have that either $d\theta (\gamma_x^ {n^ m})\to_m +\infty$ or $d\theta (\gamma_x^ {n^ m})\to_m -\infty$.
 
\end{lema}

\begin{proof} Note that $f|_A$ does not have fixed points, by Lemma \ref{fix cero}. Consider the universal cover $p:\tilde A\to A $, and take any lift $F:\tilde A\to \tilde A$ of $f|_A$. We identify $A$ with $S^1\times (0,1)$ via homeomorphism, so $\tilde A$ is identified with $\R\times (0,1)$. Note that proving the lemma is equivalent to show that for all $\tilde x\in\tilde A$ we have that either $(F^ {n^m} (\tilde x))_1 \to_m \infty$  or $(F^ {n^m} (\tilde x))_1 \to_m -\infty$. This is obtained from Lemma \ref{franks adapt}, since for $x=p(\tilde x)$ we have $f^ {n^m}(x)\to_m x$ by Lemma \ref{esrec}. 

\end{proof}

The next result will allow us to lift the action on $A$ to the universal cover $\tilde A$. We state it in general, relaxing the conditions on $h$ we assumed for this section. 

\begin{lema} \label{diag lift} Let $\langle f,h\rangle$ be a $BS(1,n)$-action on $\R^2$ where $h$ is linear and $f$ has bounded displacement and fixes $(0,0)$. Then there exist $F,H: \tilde A\to \tilde A$, lifts of $f|_A$ and $h|_A$ respectively, such that $HFH^{-1}= F^n$.
 
\end{lema}

\begin{proof} We consider $D$ the compactification of the plane with the circle of directions.  Since $h$ is linear, it extends to $D$. (For instance, if $h(x,y)=(\lambda x,\lambda y)$, then $h$ extends as the identity on $\partial D$ if $\lambda>0$, or as the antipodal map if $\lambda<0$.) On the other hand, $f$ also extends to $D$ acting as the identity on the boundary, on account of the bounded displacement hypothesis. 

We restrict $f$ and $h$ to $D\setminus \{(0,0)\}$, which is homeomorphic to the annulus with one boundary component $S^1\times(0,1]$. So we may take lifts $F, H: \R \times (0,1]\to \R \times (0,1]$ of these maps, and we can choose $F$ to be the identity on $\R\times \{1\}$.  Now $HFH^{-1}$ and $F^n$ are both lifts of the same map $f^n$ and therefore they must be equal, as they coincide (with the identity) on $\R\times \{1\}$.

\end{proof}

Next we seek to reduce Theorem \ref{teo contractivo} to the case with $\lambda = \mu$. To do this we shall take a suitable conjugation of the action, by the map described in our next result. 

\begin{lema}\label{conj} Suppose that $|\mu|<|\lambda|$.  Let $\beta= \log |\mu|/\log|\lambda|$ and consider the map \begin{displaymath}\Phi(x,y) =
\left\lbrace
\begin{array}{c l}
 (x,y^{\beta}) & \mbox{if }\,\, y\geq 0\\

(x,-|y|^{\beta}) & \mbox{if }\,\, y<0\end{array}\right.
\end{displaymath} Then $f_1= \Phi^{-1} f \Phi$ has bounded displacement.

\end{lema}

\begin{proof}  Notice that since $|\mu|<|\lambda|<1$ we get that  $0<1/\beta<1$. From this we can obtain that $\Phi^{-1}$ is quasi-Lipschitz, namely that 
$$|\Phi^{-1}(p) - \Phi^{-1}(q)|\leq |p-q|+C. $$ 
Thus we have 
$$ |f_1(p)-p| = |\Phi^{-1} f \Phi(p) - \Phi^{-1} \Phi(p)|<|f\Phi(p)-\Phi(p)| + C \leq K+C $$ since $|f(q)-q|\leq K$ for all $q$.
  
\end{proof}

\begin{cor} \label{mu=lambda} If Theorem \ref{teo contractivo} holds for $\lambda=\mu$, then it also holds for  $\lambda\neq\mu$.
 
\end{cor}

\begin{proof} We suppose without loss of generality that  $|\mu|<|\lambda|$. Then we consider the map $\Phi$ defined in Lemma \ref{conj}, which is an homeomorphism, and a direct computation shows that $$\Phi^{-1} h \Phi(x,y)=(\lambda x,\lambda y) $$ Also, $\Phi^{-1} f \Phi$ has bounded displacement by Lemma \ref{conj}. So by hypothesis we get that $\Phi^{-1} f \Phi = \id$, which implies that $f= \id$.
 
\end{proof}

The assumption that $\lambda=\mu$ yields that $d\theta(h\gamma)=d\theta(\gamma)$ for any curve $\gamma$ in $A$. 
We will also need to control the map $x\to d\theta(\gamma_x)$ for $x\in A$. Note, however, that bounded displacement for $f$ does not imply that this map is bounded, because of the singularity of $d\theta$ at the origin. For our proof of Theorem \ref{teo contractivo} it will suffice to bound it far from the origin.

\begin{lema} \label{theta bound} For every $r>0$ there is $C_r\geq 0$ such that $|d\theta(\gamma_x)|\leq C_r$ for all $x\in \R^2\setminus B((0,0),r)$.
\end{lema}

\begin{proof} Let $D$ be the compactification of $\R^2$ used in the proof of Lemma \ref{diag lift}, and recall that $f$ extends to $D$ as the identity on $\partial D$. Thus the map $x\to d\theta(\gamma_x)$ extends continuously to $D\setminus\{(0,0)\}$ by setting it to $0$ on $\partial D$. The result then follows from the compactness of $D\setminus B((0,0),r)$.

\end{proof}

We are now ready to give our second proof of Theorem \ref{teo contractivo}:

\begin{proof} By Corollary \ref{mu=lambda}, we assume that $\lambda=\mu$.  Applying Lemma \ref{bsp} (item 1) to the lifted action given by Lemma \ref{diag lift}, we deduce that for all $m\geq 1$ and $x\in A$ the curves $\gamma^{n^m}_x$ and $h^ m\gamma _{h^{-m}(x)}$ are homotopic in $A$ with fixed endpoints. Then we have 
 $$d\theta (\gamma^{n^m}_x) = d\theta(h^ m\gamma _{h^{-m}(x)}) = d\theta
(\gamma _{h^{-m}(x)})$$ recalling that $d\theta$ is $h$-invariant, because we assumed $\lambda=\mu$. Since $|\lambda|<1$, we see that for $m$ large enough $h^{-m}(x)$ is outside some fixed ball centered at $(0,0)$. So by Lemma \ref{theta bound} we have that $|d\theta (\gamma^{n^m}_x)|=|d\theta
(\gamma _{h^{-m}x})|\leq C_r$ for some $C_r\geq 0$ and all $m\geq 1$. This contradicts Lemma \ref{br} unless $f=\id$.
 
\end{proof}

\end{section}

\begin{section}{Examples for the diagonal case}\label{sec examples}

As we have seen in the previous section, if neither $|\mu|$ nor $|\lambda|$ belong to the interval $[1,n]$, then there is no faithful
action of $BS(1,n)$ by planar orientation preserving homeomorphisms with $a$ acting by $h(x,y)=(\mu x, \lambda y)$. In this section we construct examples of such actions when $1<|\lambda| \leq n$ (the case 
where $1<|\mu| \leq n$ is symmetric). The first family of examples, the product actions, work for any $\mu\neq 0$ and preserve the horizontal foliation. The second construction needs both $|\lambda|$ and $|\mu|$ in $(1,n]$, and gives examples that do not preserve a foliation by lines. We also produce an example with $|\lambda|=1$ and $|\mu|<1$. 

\subsection{Product actions}
\label{sec product action}

Recall that $BS(1,n)=\langle a,b\mid aba^{-1}=b^n\rangle$. In general, if $\varphi:BS(1,n)\to Homeo_+(\R)$ is a faithful action on the real line, then we can obtain a faithful action of $BS(1,n)$ on the plane by defining, for $g\in BS(1,n)$, $$\psi(g)(x,y)=(\varphi'(g)(x),\varphi(g)(y)),$$
where $\varphi'$ is any (not necessarily faithful) $BS(1,n)$-action on the line. This is the {\it product action} of $\varphi$ and $\varphi'$. Clearly, $\psi(b)$ above has bounded displacement if and only if $\varphi(b)$ and $\varphi'(b)$ have bounded displacement.
In order to obtain an action on the plane in which $\psi(a)$ is linear, we need to restrict our attention to actions on the line where $\varphi(a)$ and $\varphi'(a)$ are linear maps. That is, $\varphi(a)(x)=\lambda x$ and $\varphi'(a)(x)=\mu x$, for some $\lambda,\mu\in\R$.

When $\lambda=n$, then the affine action of $BS(1,n)$ on the line is a faithful action in which the map $\varphi(b)(x)=x+1$ has bounded displacement. Hence, the maps $h(x,y) = (\mu x, ny)$ and $f(x,y) = (x, y+1)$ provide an example of a a faithful planar $BS(1,n)$-action. The same holds for $f(x,y)=(x,y+c(x))$ for any continuous and bounded function $c$ that is not identically zero.

If $1<\lambda<n$, one may try to conjugate the affine action above by a homeomorphism of the line of the form $c: x\mapsto x^\alpha$, so that $c\circ\varphi(a)\circ c^{-1}(x)=n x$. The problem is that then the corresponding $\varphi(b)$  does not have bounded displacement (this is easy, and we leave it to the reader). So in this case we need to consider a different action on the line. According to \cite{rivas jgt}, up to semi-conjugacy, there is only one other candidate, which we now describe.

Let $\varphi(a)(x)=\lambda x$, initially assuming only that $\lambda>0$. For $k\in \Z$ define the {\it fundamental domains} $D_k=\{x\in \R\mid \lambda^{k}\leq x<\lambda^{k+1}\}$, and notice that $\R^+=(0,+\infty)=\bigcup_k D_k$. Let $\phi:\R\to Homeo_+(D_0)$ be a
continuous homomorphic embedding, i.e. a continuous flow on $D_0$. For $x\in D_0$ we define $\varphi(b)(x)=\phi(1)(x)$, and for $x\in D_k$ we define $\varphi(b)(x)=\varphi(a^k)\circ \phi(1/n^k)\circ \varphi(a^{-k})(x)$. This ensures that $\varphi$ extends to a representation of $BS(1,n)$ into $Homeo_+(\R^+)$, that can be further extended to the whole line by reflection, with $0$ as a global fixed point. This action is in fact faithful, see for instance \cite{rivas jgt}. The construction also works for $\lambda<0$ with minor changes: using $|\lambda|$ in the definition of $D_0$, and defining $D_k=\varphi(a^k)D_0$. (In this case $\bigcup_kD_k$ is a disconnected union of intervals, instead of $\R^+$.)

Now we assume that $1<|\lambda|\leq n$, and we want to see when $\varphi(b)$ has bounded displacement. Assume for a while that $\phi$ satisfies  \begin{enumerate}
\item[(\dag)] There is $K>0$ such that $$m|\,  \phi(1/m)(x)-x|\leq K$$ for all $x\in D_0$ and all $m\in \N$.
\end{enumerate}
Under this assumption, we show that $\varphi(b)$ has bounded displacement. Indeed, take $x\in D_k$ with $k\geq 0$, noting that the case for $k< 0$ is obtained by compactness, since $D_k\subset [-1,1]$, and the case for $-x\in D_k$ is analogous. Defining $x_0=\varphi(a^{-k})(x)=\lambda^{-k}x\in D_0$, we have 
 \begin{eqnarray*}
|\varphi(b)(x)-x|
&=&  |\varphi(a^k)\circ \phi(1/n^k)\circ\varphi(a^{-k})(x)-x|\\
&=& |\varphi(a^k)\circ \phi(1/n^k) (x_0)-\varphi(a^k)(x_0)|\\
&=& |\lambda|^k|\phi(1/n^k)(x_0)-x_0|\\
&\leq& n^k|\phi(1/n^k)(x_0)-x_0|\leq K,
\end{eqnarray*}
Then, to produce a planar $BS(1,n)$-action in our hypotheses, it is enough to consider the maps
$$h(x,y)=(\mu x, \lambda y)\, ,\quad  f(x,y)=(x,\varphi(b)(y)).$$

To finish the construction, we only need to exhibit an $\R$-action on $D_0$ satisfying $(\dag)$.  This can be done by taking $\phi$ as the flow of a $C^1$ vector field $X$ on $\bar D_0 = [1,\lambda]$ with $X(1)=X(\lambda)=0$. Let $K=\max |X|$. Then 
\begin{eqnarray*}
m|\phi(1/m)(x)-x|
&=& m \left|\int_0^{1/m} X(\phi(s)(x)) ds\right| \\
&\leq& m \int_{0}^{1/m} |X(\phi(s)(x))|ds \leq mK\frac{1}{m} = K.
\end{eqnarray*}

Notice that we can choose $X$ to be non-zero on the interior of $D_0$, so that the flow $\phi$ is conjugate to a translation on $D_0$.



\subsection{Strange rotations} \label{sec strange rot}

In this section we assume $|\lambda|, |\mu|\in(1,n]$, and by symmetry we may take $|\mu|\leq\|\lambda|$. First we consider the case where $\lambda=\mu>0$, so we have $h(x,y) = (\lambda x,\lambda y)$. Let  $\alpha = \log_{\lambda}n = \log n/\log \lambda$, and define $f$ by the following formula in polar coordinates:

$$f(r,\theta)= (r,\theta + r^{-\alpha}). $$

Observe that this map is an orientation preserving homeomorphism, though it clearly is not differentiable at the origin. To show it has bounded displacement, recall that in polar
coordinates we have $\mbox{dist}((r,\theta_1),(r,\theta_2))\leq r |\theta_1-\theta_2|$, where $\mbox{dist}$ stands for the Euclidean distance. Then we get
\begin{equation} \label{rot bound}
|f(p)-p| \leq r\cdot r^{-\alpha} = r^{1-\alpha}
\end{equation}  
where $p$ is the point with polar coordinates $(r,\theta)$. Since $1<\lambda\leq n$ we have $\alpha\geq 1$, so Equation \eqref{rot bound} implies that when $r\geq 1$ we have $|f(p)-p|\leq 1$. When $r\leq 1$ we have $|f(p)-p|<2$, for $f$ preserves the unit disk.

It is also easy to verify the Baumslag-Solitar relation in polar coordinates, noting that $h(r,\theta)=(\lambda r,\theta)$:
\begin{eqnarray*}
hfh^{-1}(r,\theta)
&=& hf(\lambda^{-1}r,\theta)   \\
&=& h(\lambda^{-1}r,\theta+(\lambda^{-1}r)^{-\alpha})  \\
&=&  (r,\theta+\lambda^{\alpha}r^{-\alpha})\\
&=& (r,\theta+nr^{-\alpha})=f^n(r,\theta)
\end{eqnarray*}
recalling that $\lambda^{\alpha}=n$ by definition of $\alpha$.

This example follows the spirit of the construction in \S \ref{sec product action}, where the fundamental domains would be $D_k:=\{(x,y)\in \R^2\mid \lambda^{k}\leq |(x,y)|<\lambda^{k+1}\}$. Still, in this case we obtain a genuinely two
dimensional action. It is faithful by the same argument used in \cite{rivas jgt} for the action on the line, noting that the action of $f$ on $D_0$ has free orbits. ($r^{-\alpha}$ is irrational for many $r\in[1,\lambda)$.)

The same map $f$ works also for $\lambda=\mu<0$, taking $\alpha=\log n/\log |\lambda|$. This time we have $h(r,\theta)=(|\lambda|r,\theta+\pi)$, and the same computation follows.

Now we adjust this example to the case when $|\mu|<|\lambda|$, both in the interval $(1,n]$. Let $\langle f,h \rangle$ be the action just constructed, i.e. $h(x,y)=(\lambda x,\lambda y)$ and $f$ as above, and let $h_1(x,y) = (\lambda x,\mu y)$. Consider $\beta=\log_{|\lambda|} |\mu| = \log |\mu|/\log|\lambda|$, and the map $\Phi(x,y) = (x,\mbox{sign}(y)|y|^{\beta})$ which was introduced in Lemma \ref{conj}. It is straightforward to compute that $h_1 = \Phi h \Phi^{-1}$, so we will just conjugate our previous example by $\Phi$. It remains to check that $f_1= \Phi f \Phi^{-1}$ has bounded displacement. Notice that $0<\beta<1$, since $1<|\mu|<|\lambda|$, thus the same argument used for Lemma \ref{conj} gives us the desired result. (This is the reverse implication to the one in Lemma \ref{conj}, where we had $|\mu|<|\lambda|<1$ and thus $\beta>1$.)


 

\subsection{Examples with eigenvalue $\pm 1$} \label{sec ejemplo con vp 1}

Let us assume that $\lambda=1$ and $0<\mu<1$, so we have $h(x,y)=(x,\mu y)$. Fix some $K>0$ and consider the sets $$D_k = [-K,K]\times [\mu^{k+1},\mu^{k})$$ for $k\in\Z$. Note that $D_k=h^kD_0$, that they are disjoint and that $\bigcup_k D_k = [-K,K]\times \R^+$, i.e. they work as the fundamental domains used for the previous examples. This construction will follow the same ideas. Let $\phi$ be a continuous flow supported on the interior of $D_0$ that preserves each horizontal line, e.g. the flow of an horizontal vector field supported on $D_0$. We can assume that $\phi(s)\to\id$ uniformly as $s\to 0$, by the same argument used for $(\dag)$ in \S \ref{sec product action}. Then we define $f(x)=\phi(1)(x)$ for $x\in D_0$, and $f(x)=h^k\circ\phi(1/n^k)\circ h^{-k}(x)$ for $x\in D_k$, which defines an action of $BS(1,n)$ by homeomorphisms of $[-K,K]\times \R^+$. This action is faithful by the argument in \cite{rivas jgt}. 

We extend $f$ to the plane by setting $f(x)=x$ for $x\notin [-K,K]\times \R^+ $. Continuity at the points in $\{\pm K\}\times\R^+$ is clear, so we must check it at the points in $[-K,K]\times\{0\}$, where the condition that $\mu<1$ comes in. Indeed, since $\mu<1$ we see that a small neighbourhood of $[-K,K]\times\{0\}$ only meets $D_k$ for $k>0$ large enough. On the other hand we note that 
\begin{equation} \label{ej 1}
\sup\{|f(x)-x|:x\in D_k\} = \sup \{ |\phi(1/n^k)(y)-y|:y\in D_0 \}
\end{equation}
since $\phi$ is horizontal and $h$ preserves the first coordinate. Hence $$ \sup\{|f(x)-x|:x\in D_k\}\to_k 0$$ as $k\to+\infty$, showing continuity. 

Note that $f$ is clearly bijective, and the same argument above proves continuity for $f^{-1}$. Equation \eqref{ej 1} also shows that $f$ has bounded displacement, in fact, this construction gives $|f(x)-x|<2K$ for all $x\in\R^2$. A similar construction can be done for $\lambda=-1$ and $-1<\mu<0$.

\end{section}

\begin{section}{Elliptic case}  

Throughout this section we let $f,h:\R^2\to\R^2$ be orientation preserving homeomorphisms with $hfh^{-1}=f^n$, and assume that $f$ has bounded displacement and $h$ is the linear map whose associated matrix is
$$R=\left(\begin{array}{cc} \cos(\theta)& \sin(\theta)\\ -\sin(\theta)& \cos(\theta)\end{array}\right)$$ where $\theta\in [0,2\pi]$. We assume that $\theta/2\pi$ is irrational, so that $h$ has infinite order. 
 In this section we prove:

\begin{teo}\label{noac} In the hypotheses above, we have $f=\id$. 
 
\end{teo}

Which implies, by conjugation by a linear map, our rigidity result in the elliptic case:  

\begin{cor} \label{noac-elip} There is no faithful $BS(1,n)$-action by planar orientation preserving homeomorphisms where the action of $b$ has bounded displacement and $a$ acts by an elliptic linear map.
\end{cor}

Next we develop the lemmas needed to prove Theorem \ref{noac}. For $r>0$, we let $S_r$ denote the circle centered at the origin
with radius $r$.

\begin{lema}\label{esf} If $x\in \fix(f)$, then $S_{|x|}\subset \fix (f)$.
 
\end{lema}

\begin{proof} If $x\in \fix(f)$, we have by Lemma \ref{bsp} (item 2), that $h^{-m}(x)\subset \fix (f)$ for all $m\geq 0$.  Since $h$ is an irrational rotation, the $h$ - backward orbit of $x$ is dense 
in  $S_{|x|}$, which implies that $S_{|x|}\subset \fix (f)$.
\end{proof}

\begin{lema}\label{fix no emp}  $\fix (f)\neq \emptyset$.
 
\end{lema}

\begin{proof} Note that \begin{equation}\label{elip-1}
f^{n^m}(0,0) = h^m (f (0,0))
\end{equation}  for all $m\geq 1$, by Lemma \ref{bsp} (item 1). If $(0,0)$ is fixed by $f$ we already have the result. If $f(0,0)$ is different from $(0,0)$, we note that it is a  recurrent point for $h$, since $h$ is an irrational rotation. It then follows from Equation \eqref{elip-1}  that $f(0,0)$ is a recurrent point for $f$, and by Brouwer's Theorem \ref{brou} there exists $x\in \fix(f)$. 
 
\end{proof}

It is worth noting that Lemmas \ref{esf} and \ref{fix no emp} do not need the bounded displacement condition on $f$. Neither does the following.

\begin{lema}\label{0 fixed} $(0,0)\in \fix(f)$.
 
\end{lema}

\begin{proof} Let $r_0 = \inf \{|x|: f(x)=x\}$, noting that $r_0<\infty$ by Lemma \ref{fix no emp}, and that it is actually the minimum of this set, by continuity of $f$.  We shall show that $r_0 = 0$, for then we would have $f(0,0) = (0,0)$, as desired. Suppose then that $r_0>0$. By Lemma \ref{esf} we get that $S_{r_0}\subset \fix(f)$, hence $f$ preserves $D_{r_0}$, the open disk of radius $r_0$, which is also $h$-invariant, as $h$ is a rotation. We can then do the same argument of Lemma \ref{fix no emp} on the disk $D_{r_0}$, observing that Brouwer's theorem holds since the open disk is homeomorphic to the plane, and we obtain a fixed point of $f$ in $D_{r_0}$. This contradicts the definition of $r_0$.

\end{proof}

By Lemma \ref{0 fixed}, both $h$ and $f$ can be restricted to the annulus $A= \R^2\setminus \{(0,0)\}$, and they define an action of $BS(1,n)$ on $A$. Just as in Section \ref{s:diag rigidity}, we see that $f|_A$ is isotopic to the identity on $A$, as it preserves orientation and both ends of the annulus, so we can assume that $\gamma_x$ lies in $A$ for each $x\in A$. We also consider $d\theta$ as in Section \ref{s:diag rigidity}, and note that since $h$ is a rotation, we have  $d\theta(h\gamma)=d\theta(\gamma)$ for any curve $\gamma$ on $A$.

\begin{lema} \label{fix anillo} $\fix (f|_A)\neq \emptyset$.
 
\end{lema}

\begin{proof} Assume that $\fix (f|_A) = \emptyset$, and choose a sequence $k_i\to_i +\infty$ such that $h^{k_i} \to \id$ uniformly in compact sets.  Then $f^{n^{k_i}}=h^{k_i}fh^{-k_i} \to f$ uniformly in compact sets, so for any $x\in A$ we can apply Lemma \ref{franks adapt} to deduce that \begin{equation}\label{theta inf elip}
|d\theta(\gamma_x^{n^{k_i}})|\to_m +\infty 
\end{equation}  by the same argument used for Lemma \ref{br}, since we assumed that $f$ has no fixed point in $A$. On the other hand, as we did when proving Theorem \ref{teo contractivo} at the end of Section \ref{s:diag rigidity}, we use Lemma \ref{diag lift} to lift the action on $A$ to the universal cover $\tilde A$, and then we apply Lemma \ref{bsp} (item 1) to this lifted action to deduce that $\gamma_x^{n^m}$ and $h^m\gamma_{h^{-m}(x)}$ are homotopic in $A$ for all $m\geq 1$ and $x\in A$. This gives us that $$d\theta (\gamma^{n^m}_x) = d\theta(h^ m\gamma _{h^{-m}(x)}) = d\theta
(\gamma _{h^{-m}(x)})$$ and since $h$ is a rotation, we have $h^{-m}(x)\in S_{|x|}$ for all $m$. Hence, by compactness, we get that $|d\theta (\gamma^{n^m}_x)|\leq C_x$ where $C_x>0$ is a bound that does not depend on  $m$. This contradicts Equation \eqref{theta inf elip}, proving the lemma.  
 
\end{proof}

For $0\leq r<s\leq+\infty$ we consider the annulus $$A_{r,s}= \{x\in \R ^2: r<|x|<s\}$$ which is clearly $h$-invariant. If $A_{r,s}$ is also $f$-invariant, the same proof of Lemma \ref{fix anillo} applies to $A_{r,s}$ instead of $A$, hence we have the following:

\begin{rem}\label{anillos}  If $A_{r,s}$ is $f$-invariant, then $\fix (f|_{A_{r,s}})\neq \emptyset$. 
 \end{rem}

We are now ready to prove Theorem \ref{noac}:

\begin{proof} Suppose $f\neq\id$, so we have $\R^2\setminus \fix (f)\neq\emptyset$. Then by Lemma \ref{esf} we can write $\R^2\setminus \fix (f) = \bigcup _{\alpha} A_{\alpha}$, where $A_{\alpha}=A_{r(\alpha),s(\alpha)}$ for some $0\leq r(\alpha)<s(\alpha)\leq+\infty$. Note that each $A_{\alpha}$ is $f$-invariant, but contains no fixed points of $f$ by definition, thus contradicting Remark \ref{anillos}.

\end{proof}

\end{section}

\begin{section}{Parabolic case}\label{sec par}

\subsection{Rigidity on the eigenspace} \label{s:rig par}

In this section we assume that $f,h:\R^2\to\R^2$ are orientation preserving homeomorphisms with $hfh^{-1}=f^n$, such that $f$ has bounded displacement, and that $h$ is the linear map whose associated matrix is
$$P=\left(\begin{array}{cc} 1& 0\\ 1& 1\end{array}\right),$$
i.e. $h(x,y)=(x,x+y)$. Denote the vertical axis by $L=\{(0,y):y\in \R\}$, noting that $L=\fix(h)$. Then we shall show:

\begin{teo}\label{par fijo} In the conditions above, we have $L\subset \fix(f)$.
 
\end{teo}

We show a non-trivial example of such an action in \S \ref{s: ex par}. Assuming Theorem \ref{par fijo} we can show the following:

\begin{cor} \label{cor-par} Let $\varphi$ be a $BS(1,n)$-action by planar orientation preserving homeomorphisms, where $\varphi(b)$ has bounded displacement and $\varphi(a)$ is a parabolic linear map. Then the eigenspace of $\varphi(a)$ is contained in $\fix(\varphi(b))$.
\end{cor}

\begin{proof} Recall that the linear map $\varphi(a)$ is parabolic when its associated matrix is conjugated in $GL(2,\R)$ to either $P$ or $$P'=\left(\begin{array}{cc} -1& 0\\ 1& -1\end{array}\right),$$ and in any case the eigenspace is a line, and can be written as $\fix(\varphi(a^2))$.  Note that the associated matrix of $\varphi(a^2)$ is conjugated to $P$ in both cases, and that $\langle b,a^2 \rangle \cong BS(1,n^2)$. Thus conjugating the action of $\langle b,a^2 \rangle$ and applying Theorem \ref{par fijo} yields the result.

\end{proof} 

In the rest of this section we prove Theorem \ref{par fijo}, through a series of lemmas.

\begin{lema}\label{lem bounded}
If the $f$-orbit of $x\in L$ is bounded, then $f^j(x)\in L$ for all $j$.
\end{lema}

\begin{proof} Note that the $h$-orbit of any $y\notin L$ is unbounded. Then, if $x\in L$ and $f^j(x)\notin L$ for some $j$, we see that $$f^{jn^k}(x)=h^kf^jh^{-k}(x)=h^kf^j(x)$$ is unbounded as $k\to \infty$.

\end{proof}

In particular, Lemma \ref{lem bounded} applies to points in $L$ that are periodic for $f$.

\begin{lema} \label{lem hat}If $x\in L$ and $f^j(x)\in L$ for some $j\neq 0$, then $f(x)\in L$.

\end{lema}

\begin{proof}  Observe that if $x$ and $f^j(x)$ belong to $L$, then 
\begin{equation} \label{f^j en L}
f^{nj}(x)=hf^jh^{-1}(x)=f^j(x).
\end{equation} So $f^{j(n-1)}(x) = x$, and by Lemma \ref{lem bounded} we deduce that  $f(x)\in L$.

\end{proof}

\begin{lema}\label{L}  For $x\in L$ there are two possibilities: 
\begin{itemize}
\item If $f(x)\in L$, then $f^{n-1}(x) = x$ and $f^j(x)\in L$ for all $j$.
\item If $f(x)\notin L$, then $f^j (x)\notin L$ for all $j\neq 0$.
\end{itemize}

\end{lema}

\begin{proof} 
For the first item, note that if $x$ and $f(x)$ are in $L$, then Equation \eqref{f^j en L} holds with $j=1$, giving $f^{n-1}(x) = x$, and then apply Lemma \ref{lem bounded}. The second item follows directly from Lemma \ref{lem hat}.

  
\end{proof}

From Lemma \ref{L} one inmediately deduces the following:

\begin{cor}  \label{cL} $ \fix (f^{n-1})\cap L = \{x\in L: f(x)\in L\}$.
 
\end{cor}

Let us introduce some notation. Put $$H^+=\{(x,y):x>0\} \quad \mbox{and} \quad H^-=\{(x,y):x<0\}$$ which are respectively the right and left components of $\R^2\setminus L$. These are the right and left {\em sides} of $L$, and moreover, we can speak about right and left sides of any oriented and properly embedded topological line on the plane, in particular of $f^j(L)$ for $j\in\Z$, which are $f^j(H^+)$ and $f^j(H^-)$ respectively. Note that since $f^j$ preserves orientation and has bounded displacement, the right side of $f^j(L)$ meets the right side of $L$, and the same for the left sides.
Let $$\alpha = \sup\{ |(f(x))_1|:x\in L\} \quad \mbox{and} \quad B=[-\alpha, \alpha]\times \R $$ where the subindex $1$ stands for the first coordinate. By definition we have $f(L)\subset B$, and note that $\alpha\leq K$, where $K$ is the bound on the displacement of $f$.
Finally, let $$C= \fix (f^{n-1})\cap L = \{x\in L: f(x)\in L\}$$ i.e. the set in Corollary \ref{cL}, and observe that it is closed. It is also $f$-invariant: If $x\in C$ then $f^j(x)\in L$ for all $j$ by Lemma \ref{L}, and from $f(x),f^2(x)\in L$ we deduce that $f(x)\in C$, while from $f^{-1}(x),x\in L$ we get $f^{-1}(x)\in C$.   We shall focus on this set, with the objective of proving that it agrees with $L$ and is contained in $\fix(f)$, thus proving Theorem \ref{par fijo}.

\begin{lema}\label{ne}  If $C=\emptyset$, then $f^j(L)\subset B$ for all $j$.
 
\end{lema}

\begin{proof} By definition, if $C=\emptyset$ then $f(L)\cap L =\emptyset$. Thus $f(L)$ is contained in either $H^+$ or $H^-$. Assume that $f(L)\subset H^+$, as the other case is symmetrical. Since $f$ preserves orientation and has bounded displacement, we see that $f(H^+)\subset H^+$, and deduce recursively that $f^{j+1}(L)$ is in the right side of $f^j(L)$ for all $j$. On the other hand, the fact that $f(L)\subset B$ implies that for $k\geq 1$ we have $$f^ {n^k}(L)= h^kf(L)\subset h^k (B) = B.$$ Then for $j\geq 1$ we pick $k$ with $j\leq n^k$ and note that $f^j(L)$ lies at the right side of $L$ and the left side of $f^{n^k}(L)$, which intersect in a region contained in $B$, since $f^{n^k}(L)\subset B$. For $j\leq -1$ we argue similarly.

\end{proof}

In fact, the situation described in Lemma \ref{ne} does not happen, as we see next:

\begin{lema}\label{cuenta} $C$ is non-empty.
 
\end{lema}

\begin{proof} Take $p\in L$ and write $f(p)=(a,b)$. Then we have $$  f^{n^k}(p) = h^kf h^{-k}(p) = h^kf(p) = (a,b+ka).$$
On the other hand:
$$ f^{n^k}(f(p)) = h^kf h^{-k}(a,b) = h^kf(a,b-ka).$$
We define $(a_k,b_k) = f(a,b-ka)$, and so we can write \begin{equation} \label{a_k}
f(f^{n^k}(p)) = f^{n^k}(f(p)) = (a_k,b_k+ka_k)
\end{equation} 
Since the displacement of $f$ is bounded by $K$, when we compare the second coordinates we must have $$|b_k - b +k(a_k-a)|\leq K,$$
and on the other hand, since $(a_k,b_k) = f(a,b-ka)$, we get $$|b_k - b + ka|\leq K.$$
From the equations above and the triangular inequality, we deduce that $|k(a_k-2a)|\leq 2K$. Hence $a_k \to 2a$ as $k\to +\infty$. 

Now we assume that $C=\emptyset$, aiming to reach a contradiction. Note that in this case $\alpha>0$, and thus we may choose $p\in L$ so that $|a|>\alpha/2$.  Since $a_k \to_k 2a$, for this $p$ we have $|a_k|>\alpha$ when $k$ is large enough. On the other hand, Lemma \ref{ne} implies that $f^{n^k+1}(p)\in B$ for all $k$, and then Equation \eqref{a_k} gives that $|a_k|\leq\alpha$ for all $k$. So we get a contradiction.

\end{proof}

\begin{lema}\label{conexo} $C$ is connected.
 
\end{lema}

\begin{proof} Suppose that $C$ is not connected and let $I$ be a bounded complementary interval of $C$ in $L$. By Lemma \ref{L} (second point) we see that $f^j(I)$ is disjoint from $L$ for all $j\neq 0$, and hence the arcs $f^j(I)$ for $j\in\Z$ are pairwise disjoint. On the other hand we have $C=Fix(f^{n-1})\cap L$, in particular the endpoints of $I$ are fixed by $f^{n-1}$, so all the closed arcs $f^{j(n-1)}(\overline{I})$ for $j\in\Z$ have these same endpoints. Then we take $k>0$ and note that $$ f^{n^k(n-1)}(I) = h^kf^{n-1}h^{-k}(I) = h^kf^{n-1}(I) $$ which cannot be disjoint from $f^{n-1}(I)$ by a standard Jordan curve argument (see Figure \ref{fl}), that we sketch as follows: 

\begin{figure}[ht]
\psfrag{i}{$I$}
\psfrag{fi}{$\,\, f^{n-1} (I)$}
\psfrag{fki}{$h^k f^{n-1} (I)$}

\begin{center}
{\includegraphics[scale=0.5]{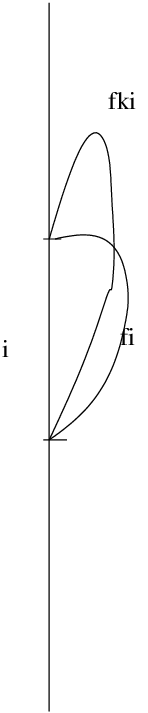}}

\caption{Lemma \ref{conexo}}\label{fl}
\end{center}
\end{figure}

We show it for $k$ large, which is sufficient. Divide the arc $f^{n-1}(\overline{I})$ as concatenation of $\gamma$ and $\delta$, where $q=\gamma(1)=\delta(0)$ has maximum horizontal coordinate, i.e. $|q_1|$ is maximum among the points in  $f^{n-1}(\overline{I})$. For the sake of notation assume that $f^{n-1}(I)$ is at the right side of $L$, the other case is symmetric. Then $h^kf^{n-1}(\overline{I})$ is contained in the band $B'=[0,q_1]\times\R$ for all $k\geq 0$. We see that $\delta$ splits $B'$ into at least two components, and $\gamma$ is contained (except for the endpoint $q$) in the lower unbounded component of $B'\setminus \delta$. If $k$ is large enough then $h^k(q)$ is in the upper unbounded component of $B'\setminus \delta$, so $h^k\gamma$ must intersect $\delta$, as it joins $h^k\gamma(0)=\gamma(0)$ with $h^k(q)$ inside $B'$. This means that $f^{n^k(n-1)}(I)=h^kf^{n-1}(I) $ interescts $f^{n-1}(I)$, thus reaching a contradiction and concluding the proof.

\end{proof}

Lemma \ref{conexo} allows us to write $$ L=I\cup C\cup J $$ where, for some $-\infty\leq a \leq b \leq +\infty$, we have $I=\{0\}\times (-\infty,a)$ and $J=\{0\}\times (b,+\infty)$. We want to show that $I$ and $J$ are empty. It shall be useful to define $$ \hat f=f^{n-1} \quad \mbox{and} \quad \hat K=(n-1)K. $$ Note that $\langle\hat f,h\rangle$ is also an action of $BS(1,n)$ by orientation preserving homeomorphisms, where $\hat f$ has bounded displacement with constant $\hat K$, and we have $C=\fix(\hat f)\cap L$. If $J\neq\emptyset$ then $\hat f(\overline{J})$ meets $\overline{J}$ exactly at $(0,b)$, so $\overline{J}\cup \hat f(\overline{J})$ is a line, which is contained in the $\hat K$-neighbourhood of $J$ by the bounded displacement condition. Then there is exactly one component of $\R^2\setminus (\overline{J}\cup \hat f(\overline{J}))$ that is contained in the $\hat K$-neighbourhood of $J$, and we denote it by $D_J$. Note that $D_J$ can also be defined as the component of $\R^2\setminus (\overline{J}\cup \hat f(\overline{J}))$ that is contained in one side of $L$, namely the side that contains $\hat f(J)$. If $I\neq\emptyset$ we can define $D_I$ analogously. These sets are sketched in Figure \ref{fl2}. 

\begin{figure}[ht]
\psfrag{i}{$I$}
\psfrag{c}{$C$}
\psfrag{j}{$J$}

\begin{center}
{\includegraphics[scale=0.5]{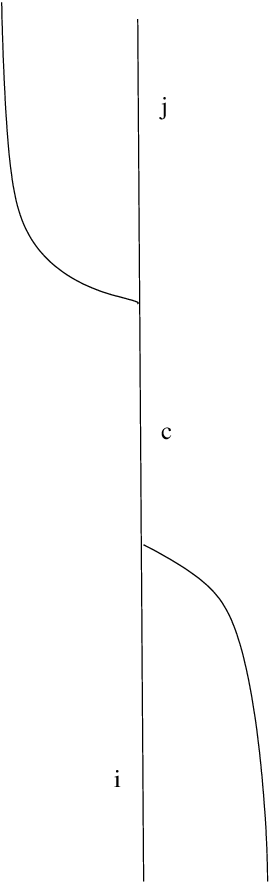}}

\caption{The sets $I$ and $J$, assuming they are non-empty, and their images by $\hat f$.}\label{fl2}
\end{center}
\end{figure}

\begin{lema}\label{J side} Assume that $J\neq\emptyset$. Then:
\begin{enumerate}
\item The sets $\hat f^j(D_J)$ are pairwise disjoint for $j\in\Z$. 
\item $\hat f^j(D_J)$ is in the same side of $L$ as $\hat f(J)$ for $j\geq 0$, and on the opposite side of $L$ for $j\leq-1$.

\end{enumerate}
 The corresponding statements hold also for $I$, in case it is non-empty.
\end{lema}

\begin{proof}
We need to recall some facts about orientation in the plane. If $\gamma$ and $\delta$ are oriented, properly embedded rays starting at a point $p\in\R^2$, the {\em sectors} defined by them are the components of $\R^2\setminus(\gamma\cup\delta)$. If $R$ is one of those sectors, we can say whether $R$ is at the right or the left of $\gamma$, according to whether $R$ lies at the right or left side of any (oriented, properly embedded) line  that extends $\gamma$ and does not meet $R$. Note that this is well defined, and that one sector is at the right of $\gamma$ and left of $\delta$, while the other is at the left of $\gamma$ and right of $\delta$. If $g$ is an orientation preserving homeomorphism that fixes $p$, then  $g$ takes the sectors defined by $\gamma$ and $\delta$ to the sectors defined by $g\gamma$ and $g\delta$, preserving the left and right sectors of each ray.

Give $\overline{J}$ the orientation of $L$ (i.e. upwards), so we can see it as a ray starting at $(0,b)$. Recalling that $(0,b)\in C\subset \fix(\hat f)$, we can observe that for $j\in\Z$ the sets $\hat f^j(J)$ are all oriented, properly embedded rays starting at $(0,b)$. The set $D_J$ is one of the sectors defined by $\overline{J}$ and $\hat f(\overline{J})$.

For simplicity of notation we assume that $\hat f(J)\subset H^-$, i.e. is at the left side of $L$, but the argument will work for the other case as well. Under this assumption we have $D_J\subset H^-$ by definition, and it is the sector at the left of $\overline{J}$ and the right of $\hat f(\overline{J})$, according to the definition above. 

First we show that $\hat f^{-1}(J)\subset H^+$: Since $\hat f^{-1}$ preserves orientation and fixes $(0,b)$, we have that $\hat f^{-1}(D_J)$ is the sector at the right of $\overline{J}$ and left of $\hat f^{-1}(\overline{J})$. Hence if $\hat f^{-1}(J)$ was in the left side of $L$, it would have to contain $H^+$, but this cannot happen: Recalling that $D_J$ is in the $\hat K$-neighbourhood of $J$ by definition, we see that $\hat f^{-1}(D_J)$ must be in a $2\hat K$-neighbourhood of $J$ by bounded displacement, but on the other hand $H^+$ is not contained in any proper neighbourhood of $J$. So we have $\hat f^{-1}(J)\subset H^+$, and moreover $\hat f^{-1}(D_J)\subset H^+$, in particular it is disjoint from $D_J$.

We will show the lemma by induction on $|j|$, adapting the argument we just made. The base case $j=0$ is trivial by the definition of $D_J$. So we take $j>0$ and assume that for $|i|<j$ the sets $\hat f^i(D_J)$ are pairwise disjoint, and contained in $H^-$ for $0\leq i <j$ and in $H^+$ for $-j<i<0$. We need to show that $\hat f^j(D_J)\subset H^-$, that $\hat f^{-j}(D_J)\subset H^+$, and that these sets are disjoint from $\hat f^i(D_J)$ for $|i|<j$ (and from each other, but that is given by the previous statements, as $H^+$ and $H^-$ are disjoint). We give the argument for $\hat f^j(D_J)$, the one for $\hat f^{-j}(D_J)$ is analogous.

Since $\hat f^{j-1}(D_J)\subset H^- $ by induction hypothesis, we see that  $\hat f^j(J) \subset H^-$, as this curve is in the boundary of $\hat f^{j-1}(D_J)$ and cannot meet $L$. We consider the line $L_j=I\cup C\cup \hat f^j(J)$ (which is oriented and properly embedded), and note that $H^+$ and $\hat f^i(D_J)$ for $|i|<j$ are in the right side of $L_j$. So we need to show that $\hat f^j(D_J)$ is in the left side of $L_j$. For this we proceed as before: We use that $\hat f$ preserves orientation and fixes $(0,b)$ to deduce that $\hat f^j(D_J)$ is the sector at the left of $\hat f^j(J)$ and the right of $\hat f^{j+1}(J)$. Thus if this sector was not contained in the left side of $L_j$, it would have to contain $I\cup C$. On the other hand, $\hat f^j(D_J)$ is contained in the $(j+1)\hat K$-neighbourhood of $J$, which does not contain $I\cup C$. Thus we have finished the induction step, observing that a similar argument works for $\hat f^{-j}(D_J)$, adjusting the notation as necessary.

Note that the argument just presented works for the case when  $\hat f(J)\subset H^+$ changing notation appropriately. Alternatively, since we did not use the map $h$, it is valid to replace $\hat f$ by $\hat f^{-1}$. The result for $I$, instead of $J$, is also analogous.

\end{proof}

Note that Lemma \ref{J side} (item 1) implies that $D_J$ cannot meet $\hat f^j(J)$ for any $j$.

\begin{lema} \label{J side 2} $ $
\begin{itemize}
\item If $J\neq\emptyset$, then $\hat f(J)$ is in the left side of $L$.
\item If $I\neq\emptyset$, then $\hat f(I)$ is in the right side of $L$.
\end{itemize}
\end{lema}

\begin{proof} The key fact will be that $h$ translates upwards in $H^+$ and downwards in $H^-$. Assume that $J\neq\emptyset$ and $\hat f(J)\subset H^+$, i.e. is in the right side of $L$. We will show that $h^k \hat f(J)$ intersects $D_J$ for $k$ large enough, which contradicts Lemma \ref{J side} because $h^k \hat f(J)=\hat f^{n^k}(J)$, as we remarked above. Consider a vertical band of the form $$B'=[0,\epsilon]\times\R$$ for $\epsilon>0$ small enough so that $\hat f(J)$ meets $L'=\{\epsilon\}\times\R$. Note that $B'\setminus \hat f(J)$ has a unique lower unbounded component, and clearly $D_J\cap B'$ is a different component of $B'\setminus \hat f(J)$. Let $\gamma$ be a small arc with $\gamma(0)\in J$ and $\gamma((0,1])$ contained in $D_J\cap B'$. Then $h^{-k}\gamma$ is contained in $B'$ for all $k$, with $h^{-k}\gamma(0)=\gamma(0)$, and since $\gamma(1)\in H^+$ we have that $h^{-k}\gamma(1)$ must belong to the lower unbounded component of $B'\setminus \hat f(J)$ for $k$ large enough. Hence $h^{-k}\gamma$ must intersect $\hat f(J)$, since it joins different components of $B'\setminus \hat f(J)$ inside $B'$. Therefore $h^k\hat f(J)$ meets $\gamma((0,1])$, which is in $D_J$.

The second point is analogous, noting that $D_I$ is unbounded in the downwards direction, in which $h$ moves the points in $H^-$.

\end{proof}

We are ready to finish the proof of Theorem \ref{par fijo}:

\begin{proof} First we show that $C=L$, i.e. that $I$ and $J$ are empty. 
Combining both items of Lemma \ref{J side 2} we see that, assuming $J\neq\emptyset$, we have $D_J=H^-\cap \hat f(H^+)$, which is to say that $\hat f^{-1}(D_J)$ is the set of the points of $H^+$ that are mapped to $H^-$ by $\hat f$. Take any $p\in \hat f^{-1}(D_J)$, and note that for $k$ large enough we have $h^{-k}(p)\notin \hat f^{-1}(D_J)$, since $p\in H^+$ and $\hat f^{-1}(D_J)$ is in the $\hat K$-neighbourhood of $J$. It follows that $\hat f h^{-k}(p)\in H^+$, and so $h^k\hat f h^{-k}(p)\in H^+$. On the other hand, applying Lemma \ref{J side} we see that the points  $p\in \hat f^{-1}(D_J)$ have $\hat f^j(p)\in H^-$ for all $j>0$.  This is a contradiction, as we just showed that $\hat f^{n^k}(p)=h^k\hat f h^{-k}(p)$ was in $H^+$. This shows that $J=\emptyset$, and the same argument works for proving that $I=\emptyset$.

So we have $C=L$, and are ready to show that $L\subset\fix(f)$. Recall that $C$ is $f$-invariant, by Lemma \ref{L} and the definition of $C$. So we can see $f|_L$ as a line homeomorphism, that preserves the orientation of $L$ since $f$ has bounded displacement. Then $f|_L^{n-1}=\hat f|_L = \id_L$, and this implies $f|_L=\id_L$ by the general theory of orientation preserving line homeomorphisms, thus finishing the proof.   

\end{proof}

\subsection{Example} \label{s: ex par}

Here we give an example of a faithful action $\langle f,h\rangle$ of $BS(1,n)$ on the plane, with $h(x,y)=(x,x+y)$ as in \S \ref{s:rig par} and $f$ an orientation preserving homeomorphism with bounded displacement. We follow the same idea we used in the constructions of \S \ref{sec examples}, setting $$D_0=\{(x,y):0\leq y < |x| \} $$ and $D_k=h^kD_0$ for all $k\in\Z$. Note that these sets are pairwise disjoint and $\bigcup_k D_k =\R^2\setminus L$, where $L$ is the vertical axis. As in \S \ref{sec examples}, given a continuous flow $\phi:\R\to Homeo_+(D_0)$ we can define an action of $BS(1,n)$ on $\bigcup_k D_k$ by the formula $f(x)=h^k\circ \phi(1/n^k)\circ h^{-k}(x)$ for $x\in D_k$. This action is faithful if $\phi(1)$ has a free orbit, as shown in \cite{rivas jgt}.

We extend $f$ to the plane by setting $f(x)=x$ for $x\in L$. Then we need to provide a flow $\phi$ on $D_0$ that has a free orbit, and so that $f$ is an homeomorphism with bounded displacement. We take $\phi:\R\to Homeo_+(D_0)$ so that:
\begin{itemize}
\item $\phi(s)$ preserves each vertical line for all $s\in\R$, and
\item $|\phi(s)(x)-x|<K$ for all $s\in\R$ and $x\in D_0$, for some constant $K>0$.
\end{itemize}
There are many examples of such flows that have free orbits, e.g. the flow defined by a non-zero vertical vector field supported in $D_0$, that vanishes on the horizontal lines $L_j = \R\times\{ y_0+jK\}$ for all $j\in\Z$ and some fixed $y_0$.      

Note that $h$ is a translation on each vertical line, so the conditions on $\phi$ imply that $f$ has bounded displacement with constant $K$. To show that $f$ is an homeomorphism, observe that since $\phi$ is supported on $D_0$ we have $$ |\phi(s)(x)-x|<|x_1| \quad \mbox{for all } x\in D_0, s\in\R $$ where $x_1$ is the first coordinate of $x$. Since $h$ is a translation on each vertical line, this implies $$|f(x)-x|<|x_1| \quad \mbox{for all } x\in \R^2\setminus L. $$ This shows that $f$ and $f^{-1}$ are continuous at the points in $L$, which are the only points where continuity was not immediate. Thus we have built the desired example.

A similar construction can be done for $h_1(x,y)=(-x,x-y)$, the other canonical form for a parabolic map, taking $$D_0=\{(x,y):x>0,\,0\leq y <x\}\cup \{(x,y):x<0,\,x< y \leq 0\}$$ which is a fundamental domain for $h_1$ acting on $\R^2\setminus L$.

\end{section}

\begin{section}{Applications to actions on the torus} \label{s:toro}

In this section we consider actions of $BS(1,n)$ by homeomorphisms of the torus $\T^2$. As before, such an action is given by two homeomorphisms $f,h:\T^2\to\T^2$ satisfying $hfh^{-1}=f^n$. In order to apply our results on planar actions, we would like to lift the action on $\T^2$ to the universal cover $\R^2\to\T^2$. So we cite the relevant results in this direction, which are based on Theorem 3 in \cite{glp}:
 
\begin{teo} \label{gl} Let $ \langle f, h \rangle$ be a faithful action of $BS(1, n)$ by homeomorphisms of $\T^2 $. Then there exists a positive integer $k$ such that $f^k$ is isotopic to the
identity and has a lift to the universal cover whose rotation set
is the single point $\{(0,0)\}$. Moreover, the set of $f^k$-fixed points is non-empty.
\end{teo}

Recall that $\langle f^k, h \rangle$ is also an action of $BS(1,n)$, so Theorem \ref{gl} allows us to restrict our study to actions where $f$ is isotopic to the identity. Moreover, we can assume that $\fix(f)\neq\emptyset$ and that $f$ has a lift $\tilde f$ to the universal cover with rotation set $\{(0,0)\}$.  We say that $\tilde f$ is the {\em irrotational lift} of $f$. Under these assumptions we can lift the action, by the following result that was proved in \cite{agx}. 

\begin{lema}\label{up}  Let $\tilde f: \R ^ 2 \to \R^2$ be the irrotational lift of $f$, and $\tilde h:\R ^ 2 \to \R ^ 2 $ be any lift of $h$.  Then we have $\tilde h \tilde f \tilde h ^{-1} = \tilde f ^n$.

\end{lema}

As an application of rigidity in the diagonalizable case, we get a new proof of a result already found in \cite{agx}: 

\begin{teo} \label{no A} There are no faithful actions $ \langle f, h \rangle$ of $BS(1,n)$ by toral homeomorphisms so that $h$ is an Anosov map with stretch factor $\lambda>n$.
 
\end{teo}

\begin{proof} Replacing $f$ by some power of it if necessary, we can lift the action to the plane using Lemma \ref{up}. Since $f$ is isotopic to the identity, its irrotational lift $\tilde f$ has bounded displacement, as it is obtained by litfing the homotopy from $\id_{\T^2}$, beginning at $\id_{\R^2}$. On the other hand, since $h$ is an Anosov map, we can choose $\tilde h$ as the lift of $h$ that is a linear map, and its canonical form  is  $$D=\left(\begin{array}{cc} \lambda & 0 \\ 0 & \lambda^{-1}\end{array}\right) $$ where $\lambda$ is the stretch factor of $h$. As we assumed that $\lambda>n$, and clearly $\lambda^{-1}<1$, we can apply Lemma \ref{lem 3,1} (up to conjugation), to deduce that $\tilde f=\id$. Thus the result follows. 
 
\end{proof}

From our work on the parabolic case, namely Theorem \ref{par fijo}, we can get the analog of Theorem \ref{no A} for Dehn twist maps:

\begin{teo} There is no faithful action $\varphi$ of $BS(1,n)$ by toral homeomorphisms so that $\varphi(a)$ is a Dehn twist map.
 
\end{teo}

\begin{proof} We consider $f,h: \R ^2\to \R ^2$ a lift of the action $\varphi$ given by Lemma \ref{up}, possibly after replacing $\varphi(b)$ by some power of it. Since $\varphi(b)$ is isotopic to the identity, the map $f$ commutes with translations by vectors in $\Z^2$. By definition of Dehn twist map, we can take $h$ as a linear map which, after conjugation in $SL(2,\Z)$, has an associated matrix of the form 
$$ P^N=\left(\begin{array}{cc} 1 & 0 \\ N & 1 \end{array}\right)$$ for $N\in \Z$, $N\neq 0$. Note that conjugation by an element of $SL(2,\Z)$ induces a conjugation of $\varphi$ in $Homeo(\T^2)$, thus we may assume that $$h(x,y)=(x,y+Nx)$$ while preserving the $\Z^2$-equivariance of $f$. 

First we note that conjugating by $D=1/N \id$ we get in the hypotheses of  Theorem \ref{par fijo}, and applying it directly we would deduce that $\varphi(b)$ fixes the points in the meridian circle $S_0=\{\overline{0}\}\times\R/\Z$. Our strategy will be to show the same for all the rational meridian circles $S_{m/l}=\{\overline{m/l}\}\times\R/\Z$ for $m,l\in\N$, thus obtaining that $\varphi(b)=\id$ by continuity. 

Let $m,l\in\N$ and consider the vertical line $L_{m/l}=\{m/l\}\times \R$. As in \S \ref{s:rig par}, we let $L=L_0$ be the vertical axis. Note that $$h^l(m/l,y)=(m/l,y+mN)=(m/l,y)+(0,mN)$$ thus the map  $g(x,y)=h^l(x,y)-(0,mN)$ fixes all the points in $L_{m/l}$. We show that $\langle f,g \rangle$ is an action of $BS(1,n^l)$:
\begin{eqnarray*}
gfg^{-1}(x,y)
&=& gf(h^{-l}(x,y)+(0,mN))   \\
&=& h^l(f h^{-l}(x,y)+(0,mN))-(0,mN) \\
&=&  h^lf h^{-l}(x,y)+h^l(0,mN)-(0,mN)\\
&=& f^{n^l}(x,y)
\end{eqnarray*}
since $f$ commutes with integer translations and $h$ is linear with $\fix(h)=L$. On the other hand let $\phi(x,y)=(x+m/l,y)$, i.e. the translation by $(m/l,0)$. Then we clearly have $\phi(L)=L_{m/l}$ and  
\begin{eqnarray*}
\phi h^l \phi^{-1}(x,y) &=& \phi h^l(x-m/l,y)    \\
&=&  \phi(x-m/l,y+lNx-mN) = (x,y+lNx-mN) \\
&=&    h^l(x,y)-(0,mN) =g(x,y) \\ 
\end{eqnarray*}
so setting $f_1=\phi^{-1}f\phi$ we get that $\langle f_1, h^l \rangle$ is the action of $BS(1,n^l)$ that we get by conjugating $\langle f,g\rangle$ by $\phi^{-1}$. Then we apply Theorem \ref{par fijo} to $\langle f_1, h^l \rangle$, (after conjugation by $1/lN \id$), to get that $L\subset\fix(f_1)$. We deduce that $f=\phi f_1 \phi^{-1}$ fixes all points in $\phi(L)=L_{m/l}$ as desired, thus finishing the proof.

\end{proof}

\end{section}

\end{document}